\documentclass[12pts,reqno]{amsart}

\usepackage[small]{caption}
\usepackage{amsmath,amsfonts,amscd,amssymb,epsf, epsfig,mathrsfs,graphicx,epsfig,euscript, eufrak,amsxtra, enumerate}
\newtheorem{theorem}{Theorem}[section]
\newtheorem{proposition}[theorem]{Proposition}

\newtheorem{remark}[theorem]{Remark}

\numberwithin{equation}{section}
\newcommand{\pa}{\partial}

\newcommand{\di}{\displaystyle}
\begin{document}
\title[ Ruelle zeta function]{ Ruelle zeta function for cofinite hyperbolic Riemann surfaces with ramification points}
\author{Lee-Peng Teo}

\address{Department of  Mathematics and Applied Mathematics, Xiamen University Malaysia, Jalan Sunsuria, Bandar Sunsuria, 43900, Sepang, Selangor, Malaysia. }

\email{lpteo@xmu.edu.my}


\date{\today}


\begin{abstract}
We consider the Ruelle zeta function $R(s)$ of a genus $g$   hyperbolic Riemann surface with $n$ punctures and $v$ ramification points.   $R(s)$ is equal to $Z(s)/Z(s+1)$, where $Z(s)$ is the Selberg zeta function. The main result of this work is the leading behavior of $R(s)$ at $s=0$. If $n_0$ is the order of the determinant of the scattering matrix $\varphi(s)$ at $s=0$, we find that
\begin{align*}
\lim_{s\rightarrow 0}\frac{R(s)}{s^{2g-2+n-n_0}}=(-1)^{\frac{A}{2}+1}(2\pi)^{2g-2+n }\tilde{\varphi}(0)^{-1} \prod_{j=1}^v m_j,
\end{align*}which says that $R(s)$ has order $2g-2+n-n_0$ at $s=0$, and its leading coefficient can be expressed in terms of $m_1$, $m_2$, $\ldots$, $m_v$,  the ramification indices at the ramification points, and $\tilde{\varphi}(0)$, the leading coefficient of $\varphi(s)$ at $s=0$. The constant $A$ is an even integer, equal to twice the multiplicity of the eigenvalue $-1$ in the scattering matrix $\Phi(s)$ at $s=1/2$, and $(-1)^{\frac{A}{2}}=\varphi\left(\frac{1}{2}\right)$.

We also consider the order   of the Ruelle zeta function  at other  integers.
\end{abstract}

 \maketitle
\section{Introduction}
In the seminal paper \cite{Selberg1956}, Selberg introduced a trace formula for a hyperbolic surface $X$, which relates the spectral trace of point-pair invariant operators to geometric quantities of the surface.  This work has very high impact to the mathematics and physics community. It has been cited close to 1500 times to date. In this paper, Selberg also introduced the zeta function
$$Z(s)=\prod_{P}\prod_{k=0}^{\infty}\left(1-p^{-s-k}\right)$$ which was named after him afterwards. In this formula, $P$ is the set of simple closed geodesics on the surface $X$, and $\log p$ is the corresponding geodesic length. It was found that this zeta function can be considered as an analogue of the Riemann zeta function, but it satisfies the "Riemann hypothesis" almost by default since the Laplacian operator on the Riemann surface is a positive definite self-adjoint operator.

A lot of works have been done subsequently to furnish the details to Selberg's paper and to generalize his results to various directions, culminating in the two-volume work by Hejhal \cite{Hejhal1976, Hejhal1983}.

The Ruelle zeta function \cite{Ruelle1976} for the hyperbolic surface $X$ is given by
$$R(s)=\prod_P\left(1-p^{-s}\right).$$
It is related to the Selberg zeta function $Z(s)$ by
$$R(s)=\frac{Z(s)}{Z(s+1)}.$$ For compact hyperbolic surfaces and surfaces with cusps, the Ruelle zeta function has been extensively studied, for example in \cite{Fried1986_1, Fried1986, Fried1986_3}, and the results have been extended to higher dimensional hyperbolic manifolds \cite{Gon1997, GonPark2008, GonPark2010, Park2009}.

In this work, we consider cofinite hyperbolic surfaces with cusps and ramification points. We first present the exact expression of the determinant of Laplacian $\Delta-s(1-s)$ in terms of the Selberg zeta function. We then use this to derive the functional equation of the Selberg zeta function and the Ruelle zeta function, and derive the explicit leading term of the Ruelle zeta function $R(s)$ at $s=0$. The result is
\begin{align*}
\lim_{s\rightarrow 0}\frac{R(s)}{s^{2g-2+n-n_0}}=(-1)^{\frac{A}{2}+1}(2\pi)^{2g-2+n }\tilde{\varphi}(0)^{-1} \prod_{j=1}^v m_j,
\end{align*}
where $m_1$, $m_2$, $\ldots$, $m_v$ are the ramification indices  at the ramification points,   and $\tilde{\varphi}(0)$ is the leading coefficient of $\varphi(s)$, the determinant of the scattering matrix $\Phi(s)$, and $A=n-\text{Tr}\,\Phi\left(\frac{1}{2}\right)$.
It is interesting to note the appearance of the term $\di \prod_{j=1}^v m_j$. When consider a Hilbert modular group, Gon \cite{Gon2014}   also obtained a  formula containing the ramification indices.

In this work, we also determine explicitly the order of $R(s)$ at all other integers. It is interesting to note that $R(s)$ can have poles at some negative integers.

\vspace{0.2cm}
\subsection*{Acknowledgements}
  This work is supported by the Ministry of   Education of Malaysia  under   FRGS grant FRGS/1/2018/STG06/XMU/01/1. We would also like to thank L. Takhtajan and J. Friedman who have given helpful comments and suggestions.

\vspace{0.2cm}

\section{The Selberg Zeta Function and the Determinant of Laplacian}
In this section, we review the relation between the Selberg zeta function and the determinant of Laplacian on a cofinite hyperbolic Riemann surface. We then present the explicit functional equation of the Selberg zeta function, and discuss the trivial zeros and poles of the Selberg zeta function.

According to uniformization theorem, for   a cofinite Riemann surface $X$, there is a finitely generated discrete  subgroup $\Gamma$ of $\text{PSL}\,(2, \mathbb{R})$  so that $X=\Gamma\backslash\mathbb{H}$. $\Gamma$ is generated by $2g$ hyperbolics elements $\alpha_1$, $\beta_1$,  $\ldots$, $\alpha_g$, $\beta_g$, $n$ parabolic elements $\kappa_1$, $\ldots$, $\kappa_n$, and $v$ elliptic elements $\tau_1, \ldots, \tau_v$ of orders $m_1, \ldots, m_v$ respectively. These generators  satisfy the nontrivial relation
\begin{align*}
\alpha_1\beta_1\alpha_1^{-1}\beta_1^{-1}\ldots \alpha_g\beta_g\alpha_g^{-1}\beta_g^{-1}\kappa_1\ldots\kappa_n\tau_1\ldots\tau_v=I,
\end{align*}where $I$ is the identity element. We say that the Riemann surface $X$ and the group $\Gamma$ are of type $(g;n;m_1, m_2, \ldots, m_v)$. Each parabolic generator $\kappa_i$ corresponds to a cusp on the Riemann surface $X$, while each elliptic generator $\tau_j$ corresponds to a ramification point, which we also call an elliptic point.

The hyperbolic area of the surface $X$ is given by
$$|X|=2\pi\left\{2g-2+n+\sum_{j=1}^v\left(1-\frac{1}{m_j}\right)\right\}.$$

 Let $\displaystyle \Delta=-y^2\left(\frac{\pa^2}{\pa x^2}+\frac{\pa^2}{\pa y^2}\right)$ be the Laplacian operator that corresponds to the hyperbolic metric
 $$ds^2=\frac{dx^2+dy^2}{y^2}$$ on $\mathbb{H}$.  The Laplacian operator acts on the space of square-integrable functions on $X$, which  correspond to functions $f$ on $\mathbb{H}$ satisfying
$$f(\gamma z)= f(z)\quad \forall \gamma\in\Gamma,\hspace{1cm} \iint\limits_X|f(z)|^2d^2z<\infty.$$

It is well-known that \cite{Iwaniec2002} the spectrum of the Laplacian operator on $X$ consists of a discrete part $0= \lambda_0< \lambda_1\leq \ldots$, as well as a continuous part which covers the interval $[1/4, \infty)$ uniformly with multiplicity $n$. Obviously, the constant functions correspond to the zero eigenvalue $\lambda_0=0$. Since $X$ is connected, the multiplicity of the zero eigenvalue is one.

For every parabolic generator $\kappa_i$, choose $\sigma_i\in \text{PSL}\,(2, \mathbb{R})$ such that $$\sigma_i^{-1}\kappa_i\sigma_i=\begin{pmatrix} 1 & 1\\0 & 1\end{pmatrix}.$$
Let $B$ be the group generated by $\begin{pmatrix} 1 & 1\\0 & 1\end{pmatrix}$ and let $\Gamma_i=\sigma_iB\sigma_i^{-1}$. Then $\Gamma_i$ stabilizes the fixed point of $\kappa_i$. Define the Eisenstein series
\begin{align*}
E_i(z,s)=\sum_{\gamma\in \Gamma_i\backslash\Gamma}\left(\text{Im}\,\sigma_i^{-1}\gamma z\right)^s.
\end{align*}Then as $y\rightarrow\infty$,
$$E_i(\sigma_j z, s)=\delta_{ij} y^{s}+\varphi_{ij}(s)y^{1-s}+\text{exponentially decaying terms}.$$

The scattering matrix $\Phi(s)$   defined by
$$\Phi(s)=\left(\varphi_{ij}(s)\right)_{1\leq i,j\leq n}$$
 is a symmetric matrix. We denote by
$\di \varphi(s)$ its determinant, i.e., $$\varphi(s)=\det\,\Phi(s).$$

Let $h:\mathbb{R}\rightarrow\mathbb{R}$ be an even function such that $h(r)$ is holomorphic in the strip $\left|\,\text{Im}\,r\,\right|\leq\frac{1}{2}+\varepsilon$ and
$$h(t)\ll \frac{1}{(1+|r|)^{2+\varepsilon}}$$
in the strip. Here $\varepsilon$ is a positive number.
Let
\begin{align*}
g(u)=\frac{1}{2\pi}\int_{-\infty}^{\infty} h(r)e^{iru}dr
\end{align*} be the Fourier transform of $h$.

The Selberg trace formula says that \cite{Venkov1982, Hejhal1983, Fischer1987, Iwaniec2002}:
\begin{theorem}[Selberg Trace Formula]
\begin{align*}
&\sum_{j=0}^{\infty}h(r_j)+\frac{1}{4\pi}\int_{-\infty}^{\infty}h(r)\frac{-\varphi'}{\varphi}\left(\frac{1}{2}+ir\right)dr\\
=&\frac{|X|}{4\pi}\int_{-\infty}^{\infty}h(r) r\tanh(\pi r)dr\\
&+\sum_{P}\sum_{\ell=1}^{\infty} \frac{1}{\di p^{\ell/2}-p^{-\ell/2}}g\left(\ell\log p\right)\log p\\
&+\sum_{j=1}^v\sum_{\ell=1}^{m_j-1}\frac{1}{\di 2m_j \sin\frac{\pi \ell}{m_j}}\int_{-\infty}^{\infty}h(r)\frac{\di \cosh \pi\left(1-2\ell/m_j\right)r}{\cosh\pi r}dr\\
&+\frac{Ah(0)}{4} +Cg(0)-\frac{n}{2\pi}\int_{-\infty}^{\infty}h(r)\psi(1+ir)dr.
\end{align*}
Here $r_j$ is defined by $\di \lambda_j=\frac{1}{4}+r_j^2$,  and $P$ runs through conjugacy classes of primitive hyperbolic elements in $\Gamma$. For each hyperbolic element in $P$, $\log p$ is the length of the corresponding closed geodesic. The function $\psi(s)$ is the logarithmic derivative of the gamma function $\Gamma(s)$. The constants $A$ and $C$ are given by
\begin{align*}
A=&n-\text{Tr}\,\Phi\left(\frac{1}{2}\right),\\
C=&-n\log 2.
\end{align*}
\end{theorem}

 The term $$\frac{1}{4\pi}\int_{-\infty}^{\infty}h(r)\frac{-\varphi'}{\varphi}\left(\frac{1}{2}+ir\right)dr$$ is the regularized trace of the continuous spectrum.

Putting
\begin{align*}
h(r)=&\frac{1}{\left(s-\frac{1}{2}\right)^2+r^2}-\frac{1}{\left(a-\frac{1}{2}\right)^2+r^2},\end{align*}so that\begin{align*}
g(u)=&\frac{1}{2s-1}e^{-|u|(s-1/2)}-\frac{1}{2a-1}e^{-|u|(a-1/2)}
\end{align*}into the Selberg trace formula
 gives the resolvent trace formula \cite{Iwaniec2002}:

\begin{equation}\label{eq2}\begin{split}
&\sum_{j=0}^{\infty}\left(\frac{1}{\lambda_j-s(1-s)}-\frac{1}{\lambda_j-a(1-a)}\right)\\&+\frac{1}{4\pi}\int_{-\infty}^{\infty}
\left(\frac{1}{\left(s-\frac{1}{2}\right)^2+r^2}-\frac{1}{\left(a-\frac{1}{2}\right)^2+r^2}\right)\frac{-\varphi'}{\varphi}\left(\frac{1}{2}+ir\right)dr\end{split}
\end{equation}
\begin{equation*}\begin{split}
=&\frac{A}{(2s-1)^2} +\frac{C}{2s-1}-\frac{n}{2s-1}\psi\left(s+\frac{1}{2}\right)\\&-\psi(s)\frac{|X|}{2\pi}+\frac{1}{2s-1}\sum_{P}\sum_{k=0}^{\infty}\frac{\log p}{p^{s+k}-1}\\
&+\frac{1}{2s-1}\sum_{j=1}^v\sum_{k=0}^{m_j-1}\frac{2k+1-m_j}{m_j^2}\psi\left(\frac{s+k}{m_j}\right)-\text{ $s$ replaced by $a$}.
\end{split}\end{equation*}

The determinant of $\Delta-s(1-s)$ is defined in the following way \cite{VenkovKalininFaddeev,Efrat1988,Koyama1991_2}. Let
\begin{align*}
\zeta(w,s)=\sum_{j=0}^{\infty}\frac{1}{(\lambda_j-s(1-s))^w}+\frac{1}{4\pi}\int_{-\infty}^{\infty}
\frac{1}{\di \left[\left(s-\tfrac{1}{2}\right)^2+r^2 \right]^w}\frac{-\varphi'}{\varphi}\left(\frac{1}{2}+ir\right)dr
\end{align*}be the spectral zeta function of $X$. This expression is well-defined when $\text{Re}\,w$ is large enough.
 It can be analytically continued to a neighbourhood of $w=0$. The zeta regularized determinant $\det (\Delta-s(1-s))$ is defined as
\begin{align}\label{eq_det}
\det (\Delta-s(1-s))=\exp\left(-\zeta_w(0,s)\right).\end{align}

By uniqueness of analytic continuation, we have
\begin{equation}\label{eq1}\begin{split}
&\frac{d}{ds}\frac{1}{2s-1}\frac{d}{ds}\left(-\zeta_w(0,s)\right)\\
=&\frac{d}{ds}\left\{\sum_{j=0}^{\infty}\left[\frac{1}{(\lambda_j-s(1-s))}-\frac{1}{(\lambda_j-a(1-a))}\right]\right.\\&\left.+\frac{1}{4\pi}\int_{-\infty}^{\infty}
\left[\frac{1}{\di  \left(s-\tfrac{1}{2}\right)^2+r^2  }-\frac{1}{\di  \left(a-\tfrac{1}{2}\right)^2+r^2  }\right]\frac{-\varphi'}{\varphi}\left(\frac{1}{2}+ir\right)dr\right\}.
\end{split}\end{equation}

Integrating \eqref{eq1}  with respect to $s$ using \eqref{eq2},
one would obtain a relation between the determinant of Laplacian $\Delta-s(1-s)$ and the Selberg zeta function up to some constants. The constants can be determined by using the Selberg trace formula with $h(r)=e^{-t(r^2+1/4)}$ to determine the asymptotic behavior of $\log\det\left(\Delta-s(1-s)\right)$ when $s\rightarrow\infty$. For compact hyperbolic surfaces, such a relation was established by D'Hoker and Phong \cite{DHokerPhong1986} and Sarnak \cite{Sarnak1987}.   Efrat \cite{Efrat1988} extended the result to cofinite hyperbolic surfaces without elliptic points. For congruence subgroups $\Gamma_0(N)$, $\Gamma_1(N)$ and $\Gamma(N)$, Koyama    obtained such a relation in \cite{Koyama1991, Koyama1991_2}. Gong \cite{Gong1995}   considered the more general case of Laplacian operators on automorphic forms of nonzero weights.

Recall the definition of the    Alekseevskii-Barnes double gamma function $\Gamma_2(s)$   \cite{Alekseevskii1889, Barnes}:
\begin{align*}
\Gamma_2(s+1)=\frac{1}{(2\pi)^{\frac{s}{2}}}e^{\frac{s}{2}+\frac{\gamma+1}{2}s^2}\prod_{k=1}^{\infty}\left(1+\frac{s}{k}\right)^{-k}e^{s-\frac{s^2}{2k} }.
\end{align*}  The following formula gives an explicit expression for the determinant of Laplacian.

\vspace{0.5cm}
\begin{theorem}\label{Determinant}If $X$ is a cofinite Riemann surface of type  $(g;n;m_1, m_2, \ldots, m_v)$, then the regularized determinant of its Laplacian is given by
\begin{align}\label{eq11}
\det\left(\Delta-s(1-s)\right)=&Z_{\infty}(s)Z(s)Z_{\text{ell}}(s)\Gamma\left(s+\frac{1}{2}\right)^{-n}(2s-1)^{\frac{A}{2}}e^{B\left(s-\frac{1}{2}\right)^2+C(s-\frac{1}{2})+D}
\end{align}where
\begin{align}\label{eq18}
Z(s)=\prod_{P}\prod_{k=0}^{\infty}\left(1-p^{-s-k}\right)
\end{align}is the Selberg zeta function of the surface $X$,
\begin{align*}
Z_{\infty}(s)=&\left(\frac{(2\pi)^{s}\Gamma_2(s)^2}{\Gamma(s)}\right)^{\frac{|X|}{2\pi}},\hspace{1cm}
Z_{\text{ell}}(s)=\prod_{j=1}^{v}\prod_{k=0}^{m_j-1}\Gamma\left(\frac{s+k}{m_j}\right)^{\frac{2k+1-m_j}{m_j}},\\
A=&n-\text{Tr}\,\Phi\left(\frac{1}{2}\right),\hspace{1cm}B=-\frac{|X|}{2\pi},\hspace{1cm}C=-n\log 2,\\
D=&\sum_{j=1}^v\frac{m_j^2-1}{6m_j}\log m_j+\frac{n}{2}\log 2\pi-\frac{|X|}{2\pi}\left(\frac{1}{2}\log2\pi-2\zeta'(-1)\right)-\frac{A}{2}\log 2.
\end{align*}
\end{theorem}

\vspace{0.5cm}
\begin{remark}\label{remark1} The constant
$$A=n-\text{Tr}\Phi\,\left(\frac{1}{2}\right)$$ is an even integer.  This can be shown as follows. Since
 $\di \Phi\,\left(\frac{1}{2}\right)$ is Hermitian, there exists a unitary matrix $U$ and a diagonal matrix $D$ such that
\begin{align*}
\Phi\left(\frac{1}{2}\right)=UDU^*.
\end{align*}
Since $\Phi(s)\Phi(1-s)=I$, we find that $\Phi(1/2)^2=I$. Hence,
\begin{align*}
I=\Phi\left(\frac{1}{2}\right)^2=UD^2U^*.
\end{align*}
This shows that $D^2=I$. Therefore, all the diagonal entries in $D$ are either 1 or $-1$. Assume that $n_+$ of them is 1 and $n_-$ of them is $-1$.
Then
$$n_++n_-=n$$
and
\begin{align*}
\text{Tr}\Phi\,\left(\frac{1}{2}\right)=n_+-n_-
\end{align*}
Hence,
$$A=n-(n_+-n_-)=2n_-$$
is an even integer. This shows that $(2s-1)^{\frac{A}{2}}$ is well-defined.\end{remark}

\vspace{0.5cm}
\begin{remark}
From the right hand side of \eqref{eq11}, we notice that on the moduli space of Riemann surfaces of type $(g;n;m_1, m_2, \ldots, m_v)$, only the Selberg zeta function $Z(s)$ depends on the moduli parameters. Hence, when one is only concerned with the variation of the determinant of Laplacian on the moduli space, such as in \cite{TakhtajanZograf1991, TakhtajanZograf2017}, one can take $Z(s)$ to be the determinant of Laplacian, up to a constant.
\end{remark}

\vspace{0.2cm}

As discuss in \cite{Efrat1988}, the determinant of Laplacian defined by \eqref{eq_det} which include the contribution of the continuous spectrum, is not invariant under the change $s\mapsto 1-s$. As discussed in \cite{Venkov1982}, the integral
$$-\frac{1}{4\pi}\int_{-\infty}^{\infty}
\left[\frac{1}{\di  \left(s-\tfrac{1}{2}\right)^2+r^2  }-\frac{1}{\di  \left(a-\tfrac{1}{2}\right)^2+r^2  }\right]\frac{\varphi'}{\varphi}\left(\frac{1}{2}+ir\right)dr$$ is equal to
\begin{align*}
-\frac{1}{2}\frac{1}{2s-1}\frac{\varphi'(s)}{\varphi(s)}+\text{a function of $s(1-s)$}.
\end{align*}Since $\varphi(s)\varphi(1-s)=1$, we have
$$\frac{\varphi'(s)}{\varphi(s)}=\frac{\varphi'(1-s)}{\varphi(1-s)}.$$ Hence, the term
$$-\frac{1}{2}\frac{1}{2s-1}\frac{\varphi'(s)}{\varphi(s)}$$ would gain an extra minus sign when $s$ is changed to $1-s$.
Therefore,
\begin{align*}
\det\left(\Delta-s(1-s)\right)=D(s(1-s))\varphi(s)^{-\frac{1}{2}},
\end{align*}where $D(s(1-s))$ is a function invariant under the change $s\mapsto 1-s$.

\vspace{1cm}

\begin{proposition}\label{functional}The functional equation of the Selberg zeta function is given by
\begin{align}\label{eq12}
Z(1-s)=\varkappa(s)Z(s),
\end{align}where
\begin{align*}
\varkappa(s)=&(-1)^{\frac{A}{2}}e^{C(2s-1)}\varphi(s)\left(\frac{(2\pi)^{2s-1}\Gamma_2(s)^2\Gamma(1-s)}{\Gamma_2(1-s)^2\Gamma(s)}\right)^{\frac{|X|}{2\pi}}
\left(\frac{\di\Gamma\left(\frac{3}{2}-s\right)}{\di\Gamma\left(s+\frac{1}{2}\right)}\right)^n\\&\times\prod_{j=1}^{v}\prod_{k=0}^{m_j-1}\left[ \sin \frac{\pi(s+k)}{m_j}\right]^{\frac{m_j-2k-1}{m_j}}.
\end{align*}
The notations are the same as in Theorem \ref{Determinant}.
\end{proposition}
\begin{proof}
From \eqref{eq11}, we have
\begin{align*}
D\left(s(1-s)\right)\varphi(s)^{-\frac{1}{2}}=&Z_{\infty}(s)Z(s)Z_{\text{ell}}(s)\Gamma\left(s+\frac{1}{2}\right)^{-n}(2s-1)^{\frac{A}{2}}e^{B\left(s-\frac{1}{2}\right)^2+C(s-\frac{1}{2})+D}.
\end{align*}
Changing $s$ to $1-s$, we find that
\begin{align*}
D\left(s(1-s)\right)\varphi(1-s)^{-\frac{1}{2}}=&Z_{\infty}(1-s)Z(1-s)Z_{\text{ell}}(1-s)\Gamma\left(\frac{3}{2}-s\right)^{-n}(1-2s)^{\frac{A}{2}}e^{B\left(s-\frac{1}{2}\right)^2-C(s-\frac{1}{2})+D}.
\end{align*}Therefore,
\begin{align*}
Z(1-s)=\varkappa(s)Z(s),
\end{align*}where
\begin{align*}
\varkappa(s)=\frac{\varphi(1-s)^{-\frac{1}{2}}}{\varphi(s)^{-\frac{1}{2}}}\frac{\di Z_{\infty}(s)Z_{\text{ell}}(s)}{\di Z_{\infty}(1-s)Z_{\text{ell}}(1-s)}(-1)^{\frac{A}{2}}e^{C(2s-1)}\left(\frac{\di\Gamma\left(\frac{3}{2}-s\right)}{\di\Gamma\left(s+\frac{1}{2}\right)}\right)^n,
\end{align*}
Now,
\begin{align*}
\frac{\varphi(1-s)^{-\frac{1}{2}}}{\varphi(s)^{-\frac{1}{2}}}=\varphi(s).
\end{align*}
\begin{align*}
\frac{\di Z_{\text{ell}}(s)}{\di  Z_{\text{ell}}(1-s)}=\prod_{j=1}^{v}\prod_{k=0}^{m_j-1}\Gamma\left(\frac{s+k}{m_j}\right)^{\frac{2k+1-m_j}{m_j}}
\prod_{j=1}^{v}\prod_{k=0}^{m_j-1}\Gamma\left(\frac{k+1-s}{m_j}\right)^{-\frac{2k+1-m_j}{m_j}}.
\end{align*}Replacing $k$ by $m_j-1-k$ in the second product, we have
\begin{align}\label{eq14}
\frac{\di Z_{\text{ell}}(s)}{\di  Z_{\text{ell}}(1-s)}=\prod_{j=1}^{v}\prod_{k=0}^{m_j-1}\Gamma\left(\frac{s+k}{m_j}\right)^{\frac{2k+1-m_j}{m_j}}
\prod_{j=1}^{v}\prod_{k=0}^{m_j-1}\Gamma\left(1-\frac{s+k}{m_j}\right)^{\frac{2k+1-m_j}{m_j}}.
\end{align}
Using the fact that
\begin{align*}
\Gamma(s)\Gamma(1-s)=\frac{\pi}{\sin\pi s}
\end{align*}and
$$\sum_{k=0}^{m-1}(2k+1-m)=0,$$we obtain
\begin{align}\label{eq13}
\frac{\di Z_{\text{ell}}(s)}{\di  Z_{\text{ell}}(1-s)}=\prod_{j=1}^{v}\prod_{k=0}^{m_j-1}\left[ \sin \frac{\pi(s+k)}{m_j}\right]^{\frac{m_j-2k-1}{m_j}}.
\end{align}
The result follows.
\end{proof}

\vspace{0.5cm}

One can deduce the orders   of the Selberg zeta function when $s$ is an integer or half-integer from the resolvent trace formula \eqref{eq2}. We can also obtain this information from the equation \eqref{eq11} and the functional equation \eqref{eq12}.

Since $\Delta$ has eigenvalue $\lambda_0=0$ with multiplicity one, it follows from the resolvent trace formula \eqref{eq2} that the Selberg zeta function $Z(s)$ has a zero of order one at $s=1$. Hence, we can write
\begin{align}\label{eq21}\det\left(\Delta-s(1-s)\right)=s(s-1)\det\!^{\prime}\left(\Delta-s(1-s)\right),\end{align} where the term $\det^{\prime}\left(\Delta-s(1-s)\right)$ excludes contribution from the zero eigenvalue, and it does not vanish when $s=0$ or $s=1$.

Now   $\Gamma_2(s)$ has a simple pole of order 1 at $s=0$ with residue one. We find that   as $s\rightarrow 0$,
\begin{align*}
Z_{\infty}(s)\sim s^{-\frac{|X|}{2\pi}}.
\end{align*}On the other hand, explicit expression for $Z_{\text{ell}}(s)$ shows that  it has a zero of order
\begin{align*}
\sum_{j=1}^v\left(1-\frac{1}{m_j}\right)
\end{align*}at $s=0$. Using
\begin{align}\label{eq17}
\frac{|X|}{2\pi}=2g-2+n+\sum_{j=1}^r\left(1-\frac{1}{m_j}\right),
\end{align}we conclude from \eqref{eq12} that if the order of $\varphi(s)$ at $s=0$ is $n_0$, then the order of $Z(s)$ at $s=0$ is $2g-1+n-n_0$.

\vspace{0.5cm}
\begin{proposition}\label{Pzero}For a confinite Riemann surface of type $(g;n;m_1, \ldots, m_v)$, its Selberg zeta function $Z(s)$ has a zero of order one at $s=1$. The order of $Z(s)$ at $s=0$ is $2g-1+n-n_0$, where $n_0$ is the order of $\varphi(s)$ at $s=0$.

\end{proposition}

The infinite product expression for the Selberg zeta function \eqref{eq18} is convergent when $\text{Re}\,s>1$. Hence, we can deduce from the relation
\begin{align*}
Z(s)=\varkappa(1-s)Z(1-s)
\end{align*}the zeros of $Z(s)$ on the half plane $\text{Re}\,s<0$.

\begin{theorem}\label{negativezero}
Let $X$ be a  confinite Riemann surface of type $(g;n;m_1, \ldots, m_v)$ and let $Z(s)$ be its Selberg zeta function $Z(s)$.
\begin{enumerate}
\item[(a)]
  $Z(s)$ has a pole of order $n$ at $s=-1/2, -3/2, -5/2, \ldots$.
\item[(b)] For $k=1, 2, 3, \ldots$,  $Z(s)$ has a zero of order
$$(2k+1)(2g-2+n)+2\sum_{j=1}^v\left(k-\left\lfloor  \frac{k}{m_j}\right\rfloor\right)$$ at $s=-k$.\end{enumerate}$Z(s)$ does not have other zeros and poles on the half plane $\text{Re}\,s<0$.
\end{theorem}
\begin{proof}Since $Z(s)$ is regular when $\text{Re}\,s>1$, all the zeros and poles of $Z(s)$ on the half-plane $\text{Re}\,s<0$ comes from
\begin{align*}
\varkappa(1-s)=\varphi(1-s)\frac{\di Z_{\infty}(1-s)Z_{\text{ell}}(1-s)}{\di Z_{\infty}(s)Z_{\text{ell}}(s)}(-1)^{\frac{A}{2}}e^{C(1-2s)}\left(\frac{\di\Gamma\left(\frac{1}{2}+s\right)}{\di\Gamma\left(\frac{3}{2}-s\right)}\right)^n.
\end{align*}
When $\text{Re}\,s<0$, $\varphi(1-s)$, $Z_{\infty}(1-s)$, $Z_{\text{ell}}(1-s)$ and $\di\Gamma\left(\frac{3}{2}-s\right)$ do not have poles or zeros.

$\di\Gamma\left(\frac{1}{2}+s\right)$ has poles of order one at $s=-1/2, -3/2, -5/2, \ldots$, which give rise to poles of order $n$ at $s=-1/2, -3/2, -5/2, \ldots$ for $Z(s)$.

$Z_{\infty}(s)$ has poles of order $\di(2k+1)\frac{|X|}{2\pi}$ at $s=-k$, where $k=1, 2, \ldots$, while
the order of $Z_{\text{ell}}(s)$ at $s=-k$ is
$$-\sum_{j=1}^v\frac{2\left(k-\di\left\lfloor  \frac{k}{m_j}\right\rfloor m_j \right)+1-m_j}{m_j}.$$

Hence,  the order of  $Z(s)$ at $s=-k$ is
\begin{align*}
s_k=&(2k+1)\left[2g-2+n+\sum_{j=1}^v\left(1-\frac{1}{m_j}\right)\right]+\sum_{j=1}^v\frac{2\left(k-\di\left\lfloor  \frac{k}{m_j}\right\rfloor m_j \right)+1-m_j}{m_j}\\
=&(2k+1)(2g-2+n)+2\sum_{j=1}^v\left(k-\left\lfloor  \frac{k}{m_j}\right\rfloor\right).
\end{align*}This gives the desired formula for the order of $Z(s)$ at $s=-k$. Finally let us prove that this is a nonnegative integer.
Notice that
\begin{align*}
&(2k+1)(2g-2+n)+2\sum_{j=1}^v\left(k-\left\lfloor  \frac{k}{m_j}\right\rfloor\right)\\
=&2\left[k(2g-2+n)+\sum_{j=1}^v\left(k-\left\lfloor  \frac{k}{m_j}\right\rfloor\right)\right]+2g-2+n.\end{align*}
Now
\begin{align*}
k(2g-2+n)+\sum_{j=1}^v\left(k-\left\lfloor  \frac{k}{m_j}\right\rfloor\right)
\geq &k(2g-2+n)+\sum_{j=1}^v\left(k- \frac{k}{m_j} \right)\\
=&k\left[2g-2+n+\sum_{j=1}^v\left(1- \frac{1}{m_j} \right)\right]
>0.
\end{align*}Since the left hand side  is an integer, we must have
$$k(2g-2+n)+\sum_{j=1}^v\left(k-\left\lfloor  \frac{k}{m_j}\right\rfloor\right)\geq 1.$$ Together with
$$2g-2+n\geq -2, $$ we find that
$$s_k\geq 0.$$\end{proof}

\vspace{0.5cm} Finally, we would like to give an explicit   expression for $\det'\Delta$ which is of particular interest.

\begin{theorem}\label{Main2}
If $X$ is a cofinite Riemann surface of type  $(g;n;m_1, m_2, \ldots, m_v)$, then
\begin{align}\label{eqnew}
\det\!^{\prime}\Delta= c_1Z'(1),
\end{align}
where
\begin{align*}
c_1=&2^{n-\frac{A}{2}}(2\pi)^{\frac{|X|}{4\pi}}\mathscr{E}\prod_{j=1}^{v}\prod_{k=1}^{m_j-1}\Gamma\left(\frac{k}{m_j}\right)^{\frac{2k-1-m_j}{m_j}},
\end{align*}
and
$$\mathscr{E}=\exp\left\{\sum_{j=1}^v\frac{m_j^2-1}{6m_j}\log m_j+\frac{|X|}{2\pi}\left(2\zeta'(-1)-\frac{1}{4}\right) \right\}.$$
\end{theorem}
\begin{proof}
Since    the Selberg zeta function $Z(s)$ has a simple zero at $s=1$,  we can write \begin{align*}
Z(s)=&(s-1)\hat{Z}(s),
\end{align*}where  $\hat{Z}(1)$ is nonzero and finite.

Using \eqref{eq11} and \eqref{eq21}, we have
\begin{align*}
&s(s-1)\det\!^{\prime}\left(\Delta-s(1-s)\right)\\
=&Z_{\infty}(s)(s-1)\hat{Z}(s)Z_{\text{ell}}(s)\Gamma\left(s+\frac{1}{2}\right)^{-n}(2s-1)^{\frac{A}{2}}e^{C(s-\frac{1}{2})}e^{B\left(s-\frac{1}{2}\right)^2+D}.
\end{align*}Comparing the leading terms of both sides when $s\rightarrow 1$ give
\begin{align}\label{eq23}
\det\!^{\prime}\Delta=c_1\hat{Z}(1)= c_1Z'(1),
\end{align}
where
\begin{align*}
c_1=&Z_{\infty}(1)Z_{\text{ell}}(1)\Gamma\left( \frac{3}{2}\right)^{-n}e^{\frac{B}{4}+\frac{C}{2}+D}\\
=&(2\pi)^{\frac{|X|}{2\pi}}\prod_{j=1}^{v}\prod_{k=0}^{m_j-1}\Gamma\left(\frac{1+k}{m_j}\right)^{\frac{2k+1-m_j}{m_j}}\frac{2^n}{\pi^{\frac{n}{2}}}
e^{\frac{B}{4}+\frac{C}{2}+D}\\
=&2^{n-\frac{A}{2}}(2\pi)^{\frac{|X|}{4\pi}}\mathscr{E}\prod_{j=1}^{v}\prod_{k=1}^{m_j-1}\Gamma\left(\frac{k}{m_j}\right)^{\frac{2k-1-m_j}{m_j}},
\end{align*}with
$$\mathscr{E}=\exp\left\{\sum_{j=1}^v\frac{m_j^2-1}{6m_j}\log m_j+\frac{|X|}{2\pi}\left(2\zeta'(-1)-\frac{1}{4}\right) \right\}.$$

\end{proof}

\vspace{0.5cm}
\section{The    Ruelle zeta function }
Recall that the Ruelle zeta function of a hyperbolic Riemann surface $X$ is defined as
\begin{align*}
R(s)=\prod_{P}\left(1-p^{-s}\right).
\end{align*}
It can be expressed in terms of the Selberg zeta function:
\begin{align}\label{eqR}
R(s)=\frac{Z(s)}{Z(s+1)}.
\end{align}
The behavior of the Ruelle zeta function at $s=0$ has been of interest  \cite{Fried1986,Fried1986_3,GonPark2008,Park2009,GonPark2010,DyatlovZworski2017}. The order of singularity for compact hyperbolic surfaces  has been determined. More recently, Dyatlov and Zworski have shown that for a compact negatively curved oriented
surface, the Ruelle zeta function vanishes to the order given by the negative of the Euler characteristic at $s=0$. Here we want to derive corresponding result for hyperbolic surfaces with elliptic points.

First we consider the functional equation for the Ruelle zeta function, which generalizes the result of \cite{Gon1997} to surfaces with cusps and ramification points.
\begin{theorem}The functional equation of the Ruelle zeta function is given by
\begin{align}\label{eq15}
R(s)R(-s)=& \left(\varphi(s)\varphi(-s)\right)^{-1}\frac{(4\sin^2\pi s)^{2g-2+n}}{(4s^2-1)^n} \prod_{j=1}^v\left(\frac{\quad\sin\pi s\quad}{\di\sin\frac{\pi s}{m_j}}\right)^{2}.
\end{align}
\end{theorem}
\begin{proof}
From the functional equation for the Selberg zeta function, we have
\begin{align*}
Z(1-s)=&\varkappa(s)Z(s),\\
Z(-s)=&\varkappa(s+1)Z(1+s).
\end{align*}Taking quotient of the first expression to the second one, we find that
\begin{align*}
R(s)R(-s)=&\frac{\varkappa(s+1)}{\varkappa(s)}.
\end{align*}By \eqref{eq13}, we have
\begin{align*}
\frac{\varkappa(s+1)}{\varkappa(s)}=&\frac{\varphi(1+s)}{\varphi(s)}\frac{Z_{\infty}(s+1)}{Z_{\infty}(s)}\frac{Z_{\infty}(1-s)}{Z_{\infty}(-s)}
\frac{Z_{\text{ell}}(s+1)}{Z_{\text{ell}}(s)}\frac{Z_{\text{ell}}(1-s)}{Z_{\text{ell}}(-s)}e^{2C}\\&\times
\left[\frac{\di\Gamma\left(\frac{1}{2}-s\right)\Gamma\left(s+\frac{1}{2}\right)}{\di\Gamma\left(s+\frac{3}{2}\right)\Gamma\left(\frac{3}{2}-s\right)}\right]^n.
\end{align*}Using the functional equation
$$\Gamma_2(s)=\Gamma(s)\Gamma_2(s+1),$$ we find that
\begin{align*}
\frac{Z_{\infty}(s+1)}{Z_{\infty}(s)}\frac{Z_{\infty}(1-s)}{Z_{\infty}(-s)}
=&\left[\frac{(2\pi)^2}{\Gamma(s)\Gamma(s+1)\Gamma(-s)\Gamma(1-s)}\right]^{\frac{|X|}{2\pi}}\\
=&\left(-4\sin^2\pi s\right)^{\frac{|X|}{2\pi}}.
\end{align*}On the other hand,
\begin{align*}
\frac{\di\Gamma\left(\frac{1}{2}-s\right)\Gamma\left(s+\frac{1}{2}\right)}{\di\Gamma\left(s+\frac{3}{2}\right)\Gamma\left(\frac{3}{2}-s\right)}
=&\frac{1}{\di\left(\frac{1}{2}+s\right)
\left(\frac{1}{2}-s\right)}=\frac{4}{1-4s^2}.
\end{align*}Finally,
\begin{align*}
\frac{Z_{\text{ell}}(s+1)}{Z_{\text{ell}}(s)}=&\prod_{j=1}^{v}
 \frac{\di\Gamma\left(\frac{s}{m_j}\right)^{\frac{m_j-1}{m_j}}\Gamma\left(\frac{s+m_j}{m_j}\right)^{\frac{m_j-1}{m_j}}}{\di\prod_{k=1}^{m_j-1}
 \Gamma\left(\frac{s+k}{m_j}\right)^{\frac{2}{m_j}}}.
\end{align*}Using the identity (see \cite{GradshteynRyzhik})
\begin{align}\label{eq24}
\Gamma(s)=(2\pi)^{\frac{1-m}{2}}m^{s-\frac{1}{2}}\prod_{k=0}^{m-1}\Gamma\left(\frac{s+k}{m}\right),
\end{align}
we find that
\begin{align*}
\frac{Z_{\text{ell}}(s+1)}{Z_{\text{ell}}(s)}=&\prod_{j=1}^{v}
\frac{\di\Gamma\left(\frac{s}{m_j}\right)^{\frac{m_j+1}{m_j}}\Gamma\left(\frac{s+m_j}{m_j}\right)^{\frac{m_j-1}{m_j}}}
{\di (2\pi)^{\frac{m_j-1}{m_j}}m_j^{\frac{1-2s}{m_j}}\Gamma(s)^{\frac{2}{m_j}}}.
\end{align*}
Hence,
\begin{align*}
&\frac{Z_{\text{ell}}(s+1)}{Z_{\text{ell}}(s)}\frac{Z_{\text{ell}}(1-s)}{Z_{\text{ell}}(-s)}\\=& \prod_{j=1}^{r}
\frac{\di\Gamma\left(\frac{s}{m_j}\right)^{\frac{m_j+1}{m_j}}\Gamma\left(-\frac{s}{m_j}\right)^{\frac{m_j+1}{m_j}}\Gamma\left(\frac{s+m_j}{m_j}\right)^{\frac{m_j-1}{m_j}}
\Gamma\left(\frac{-s+m_j}{m_j}\right)^{\frac{m_j-1}{m_j}}}
{\di (2\pi)^{\frac{2m_j-2}{m_j}}m_j^{\frac{2}{m_j}}\Gamma(s)^{\frac{2}{m_j}}\Gamma(-s)^{\frac{2}{m_j}}}\\
=&\prod_{j=1}^v\frac{\di \left(\sin \pi s\right)^{\frac{2}{m_j}}}{\di (-4)^{\frac{m_j-1}{m_j}}\sin^2\frac{\pi s}{m_j}}.
\end{align*}
Gathering the terms, and using \eqref{eq17} and $\varphi(1+s)=\varphi(-s)^{-1}$,
we find that
\begin{align*}
R(s)R(-s)=& \left(\varphi(s)\varphi(-s)\right)^{-1}\frac{(4\sin^2\pi s)^{2g-2+n}}{(4s^2-1)^n} \prod_{j=1}^v\left(\frac{\quad\sin\pi s\quad}{\di\sin\frac{\pi s}{m_j}}\right)^{2}.
\end{align*}
\end{proof}
Let $$\varphi(s)=s^{n_0}\tilde{\varphi}(s),$$ where $\tilde{\varphi}(0)$ is   nonzero and finite.
As $s\rightarrow 0$,
$$\tilde{\varphi}(s)\sim s^{n_0}\tilde{\varphi}(0)+\text{higher order terms},$$
$$\sin \pi s\sim \pi s+ \text{higher order terms}.$$
This shows that
$$R(s)R(-s)\sim (-s^2)^{-n_0}(-4\pi^2 s^2)^{2g-2+n}\tilde{\varphi}(0)^{-2}\prod_{j=1}^v m_j^2\;+ \text{higher order terms}.$$
Hence,
\begin{align*}
\lim_{s\rightarrow 0}\frac{R(s)}{s^{2g-2+n-n_0}}=\pm (2\pi)^{2g-2+n}\tilde{\varphi}(0)^{-1}\prod_{j=1}^v m_j.
\end{align*}This only determine the leading term up to a plus or minus sign. To determine this plus or minus sign, we use a different approach.

\vspace{0.5cm}
\begin{theorem}\label{Main}
If $X$ is a cofinite Riemann surface of type  $(g;n;m_1, m_2, \ldots, m_v)$, then as $s\rightarrow 0$, the leading behavior of the Ruelle zeta function is given by
\begin{align*}
\lim_{s\rightarrow 0}\frac{R(s)}{s^{2g-2+n-n_0}}=(-1)^{\frac{A}{2}+1}(2\pi)^{2g-2+n } \tilde{\varphi}(0)^{-1}\prod_{j=1}^v m_j,
\end{align*}where $n_0$ is the order of $\varphi(s)$ at $s=0$, and
$$\tilde{\varphi}(s)=\frac{\varphi(s)}{s^{n_0}}.$$
\end{theorem}
\begin{proof}
Using \eqref{eqR}, we find that
\begin{align*}
\lim_{s\rightarrow 0}\frac{R(s)}{s^{2g-2+n-n_0}}=&\lim_{s\rightarrow 0}\frac{Z(s)}{s^{2g-1+n-n_0}}\frac{s}{Z(s+1)}\\
=&-\lim_{s\rightarrow 0}\frac{Z(s)}{s^{2g-1+n-n_0}}\frac{s}{Z(1-s)}\\
=&-\lim_{s\rightarrow 0}\frac{1}{s^{2g-2+n-n_0}\chi(s)}.
\end{align*}Now as $s\rightarrow 0$, one obtains from the proof of Proposition \ref{functional} that
\begin{align*}
\chi(s)\sim s^{n_0}\tilde{\varphi}(0)(-1)^{\frac{A}{2}}\left(\frac{1}{2\pi s}\right)^{\frac{|X|}{2\pi}}\prod_{j=1}^v \left(\frac{m_j}{s}\right)^{\frac{1-m_j}{m_j}}\prod_{k=1}^{m_j-1}\Gamma\left(\frac{k}{m_j}\right)^{\frac{2k+1-m_j}{m_j}}
\prod_{k=0}^{m_j-1}\Gamma\left(\frac{k+1}{m_j}\right)^{-\frac{2k+1-m_j}{m_j}}.
\end{align*}
Hence,
\begin{align*}
\lim_{s\rightarrow 0} \left(s^{2g-2+n-n_0}\chi(s)\right)=& \tilde{\varphi}(0)(-1)^{\frac{A}{2}}
\left(\frac{1}{2\pi }\right)^{\frac{|X|}{2\pi}}\prod_{j=1}^v m_j^{\frac{1-m_j}{m_j}}\prod_{k=1}^{m_j}\Gamma\left(\frac{k}{m_j}\right)^{\frac{2}{m_j}}.
\end{align*}
Putting $s=1$ in \eqref{eq24} give
$$ \prod_{k=1}^{m_j}\Gamma\left(\frac{k}{m_j}\right)^{\frac{2}{m_j}}=(2\pi)^{\frac{m_j-1}{m_j}}m_j^{-\frac{1}{m_j}}.$$
Hence,
\begin{align*}
\lim_{s\rightarrow 0}\frac{R(s)}{s^{2g-2+n-n_0}}=(-1)^{\frac{A}{2}+1}(2\pi)^{2g-2+n } \tilde{\varphi}(0)^{-1}\prod_{j=1}^v m_j.
\end{align*}
\end{proof}

We would like to thank J. Friedman for suggesting us to prove Theorem \ref{Main} using the functional equation of the Selberg zeta function. We would also like to remark that Fried \cite{Fried1986_3} has considered the leading term of the Ruelle zeta function at $s=0$ up to the plus minus sign, for a cocompact hyperbolic surface, using the functional equation. He did not obtain the term which contains the product of the ramification indices.

\vspace{0.5cm}

\begin{remark}
Recall from Remark \ref{remark1} that $\di\frac{A}{2}$ is the multiplicity of the eigenvalue $-1$ of the Hermittian and unitary matrix $\di\Phi\left(\frac{1}{2}\right)$. Therefore, the sign of the leading coefficient of $R(s)$ tells us whether the matrix $\di\Phi\left(\frac{1}{2}\right)$ has an even or odd number of eigenvalue $-1$. This is equivalent to
$$(-1)^{\frac{A}{2}}=\det\Phi\left(\frac{1}{2}\right)=\varphi\left(\frac{1}{2}\right).$$
\end{remark}

\vspace{0.5cm}

It is well known that for $1\leq i, j\leq n$, $\varphi_{ij}(s)$ has at most   a simple pole at $s=1$. Hence, $\varphi(s)$ has at most a pole of order $n$ at $s=1$. This implies that $n_0\leq n$, so the order of the Ruelle zeta function at $s=0$ is $2g-2+n-n_0\geq 2g-2$. Therefore, $R(s)$ has at most a pole of order 2 at $s=0$.

For  the modular group $\text{PSL}\,(2,\mathbb{Z})$,
 the surface $  \text{PSL}\,(2,\mathbb{Z})\backslash\mathbb{H}$ is a surface of type $(0; 1; 2, 3)$ and we know that
 $$\varphi(s)=\sqrt{\pi}\frac{\Gamma\left(s-\frac{1}{2}\right)}{\Gamma(s)}\frac{\zeta(2s-1)}{\zeta(2s)}.$$ Hence, $\varphi(1/2)=-1$,
 $n_0=1$ and $\tilde{\varphi}(0)=\pi/3$. Hence, the Ruelle zeta function $R(s)$ has a pole of order $2$ at $s=0$ and
 \begin{align*}
 \lim_{s\rightarrow 0}\left(s^2R(s)\right)=\frac{9}{\pi^2}.
 \end{align*}

 It is interesting that the order of vanishing or singularity of the Ruelle zeta function $R(s)$ at $s=0$ captures  information of the underlying hyperbolic surface. Its leading coefficient contains the information about  the orders of the elliptic generators. One can make an analogy of this result to the Birch and Swinnerton-Dyer conjecture, which conjectures the leading term of the Hasse-Weil $L$-function of an elliptic curve $E$ at $s=1$ to be related to the rank of the abelian group of points of $E$ and other arithmetic data of $E$.

\vspace{0.5cm} Let $Z(s)=s^{2g-1+n-n_0}\tilde{Z}(s)$, where $\tilde{Z}(0)$ is nonzero and finite.
From the proof of Theorem \ref{Main}, we deduce that
\begin{align*}
\frac{Z'(1)}{\tilde{Z}(0)}=&-\lim_{s\rightarrow 0}\left(s^{2g-2+n-n_0}\chi(s)\right)=(-1)^{\frac{A}{2}+1}(2\pi)^{-2g+2-n } \tilde{\varphi}(0)\prod_{j=1}^v m_j^{-1}.
\end{align*}
Together with Theorem \ref{Main2}, we find that

\begin{theorem}\label{Main3}
If $X$ is a cofinite Riemann surface of type  $(g;n;m_1, m_2, \ldots, m_v)$, then
\begin{align}\label{eqnew}
\det\!^{\prime}\Delta= c_0\tilde{Z}(0),
\end{align}
where
\begin{align*}
c_0=&(-1)^{\frac{A}{2}+1} 2^{n-\frac{A}{2}}(2\pi)^{-\frac{|X|}{4\pi}}\mathscr{E} \tilde{\varphi}(0) \prod_{j=1}^v\left(\frac{1}{m_j}\right)^{\frac{m_j-1}{m_j}}\prod_{j=1}^v\prod_{k=1}^{m_j-1}\Gamma\left(\frac{k}{m_j}\right)^{\frac{2k+1-m_j}{m_j}},
\end{align*}
and
$$\mathscr{E}=\exp\left\{\sum_{j=1}^v\frac{m_j^2-1}{6m_j}\log m_j+\frac{|X|}{2\pi}\left(2\zeta'(-1)-\frac{1}{4}\right) \right\}.$$
\end{theorem}

 \vspace{0.5cm}
Finally, we consider the order of the Ruelle zeta function $R(s)$ at other integers.
\begin{theorem}
Let $X$ be a cofinite Riemann surface of type  $(g;n;m_1, m_2, \ldots, m_v)$, and let $R(s)$ be its Ruelle zeta function.
\begin{enumerate}
\item[(a)] $R(s)$ has a simple zero at $s=1$.
\item[(b)] $R(s)$ has a zero of order $2(2g-2+n+v)+n_0-1$ at $s=-1$.
\item[(c)] For $k=2, 3, 4, \ldots$, the order of $R(s)$ at $s=k$ is 0.
\item[(d)] For $k=2, 3, 4, \ldots$, the order of $R(s)$ at $s=-k$ is
\begin{align*}
 o_k= 2\left[2g-2+n+v -\sum_{j=1}^v\varkappa_{j}(k)\right],
\end{align*}
where $$\varkappa_j(k)=\begin{cases} 1,\quad &\text{if}\;\;m_j\mid k\\0, &\text{otherwise}\end{cases}$$$o_k$ is an even integer not smaller than $-4$.
\end{enumerate}
\end{theorem}
\begin{proof}
Using \eqref{eqR}, Proposition \ref{Pzero}  and the fact that $Z(s)$ is regular at $s=2$, we immediately obtain that $R(s)$ has a simple zero at $s=1$.  Theorem \ref{negativezero} says that $Z(s)$ has a zero of order
$$3(2g-2+n)+2v$$at $s=-1$, and Proposition \ref{Pzero} says the order of $Z(s)$ at $s=0$ is $2g-1+n-n_0$. Hence, the order of $R(s)$ at $s=-1$ is
$$3(2g-2+n)+2v-(2g-1+n-n_0)=2(2g-2+n+v)+n_0-1.$$Since $2g-2+n+v>0$, we find that $2g-2+n+v\geq 1$ and hence $R(s)$ has a zero of order $2(2g-2+n+v)+n_0-1$ at $s=-1$.

For $k=2, 3, 4, \ldots$,  since $Z(s)$ and $Z(s+1)$ are both regular and nonzero at $s=k$, so is $R(s)$. The functional equation \eqref{eq15} implies that the order of $R(s)$ at $s=-k$ is
\begin{align*}
 o_k= 2\left[2g-2+n+v -\sum_{j=1}^v\varkappa_{j}(k)\right],
\end{align*}
where $$\varkappa_j(k)=\begin{cases} 1,\quad &\text{if}\;\;m_j\mid k\\0, &\text{otherwise}\end{cases}.$$Notice that $o_k$ is an even integer and
\begin{align*}
2(2g-2+n)\leq  o_k\leq 2(2g-2+n+v).
\end{align*}Since $g, n\geq 0$, the minimum possible value  of $o_k$  is $-4$, and this can happen for surfaces with $g=0$, $n=0$ when $k$ is a common multiple of $m_1$, $m_2$, $\ldots$, $m_v$.

 \end{proof}

\begin{remark}
Since $R(s)$ is regular and nonzero when $s>1$, the functional equation \eqref{eq15} implies that $R(s)$ does not have zeros or poles when $s<-1$ and $s$ is not an integer.
\end{remark}

In this work, we only consider two- (real) dimensional manifolds with conical singularities. This has important applications especially to number theory since the Riemann surface $\Gamma\backslash\mathbb{H}$ with $\Gamma$ a congruence subgroup, is of this type. In principal, one can also consider higher dimensional real hyperbolic manifolds, generalizing the results in \cite{Park2009,GonPark2010} to orbifolds.

\end{document}